\documentclass[twoside,12pt,leqno]{amsproc}
\usepackage{amssymb,latexsym,enumerate,verbatim,tikz}
\usepackage[pagebackref]{hyperref}
\usepackage{amsrefs,booktabs}  
\hypersetup{citecolor=purple, linkcolor=blue, colorlinks=true}
\numberwithin{table}{section}

\theoremstyle{plain}
\newtheorem{theorem}{Theorem}[section]
\newtheorem{lemma}[theorem]{Lemma}
\newtheorem{proposition}[theorem]{Proposition}
\newtheorem{corollary}[theorem]{Corollary}
\newtheorem{conjecture}[theorem]{Conjecture}
\theoremstyle{definition} 
\newtheorem{definition}[theorem]{Definition}
\newtheorem{remark}[theorem]{Remark}

\renewcommand{\geq}{\geqslant}
\renewcommand{\leq}{\leqslant}
\renewcommand{\ge}{\geqslant}
\renewcommand{\le}{\leqslant}

\newcommand{\A}{\mathsf{A}}
\newcommand{\Alt}{\mathrm{Alt}}
\newcommand{\calE}{\mathcal{E}}

\newcommand{\calT}{\mathcal{T}}
\newcommand{\Char}{\mathrm{char}}
\newcommand{\D}{\mathsf{D}}
\newcommand{\eps}{\varepsilon}

\newcommand{\GL}{\mathrm{GL}}
\newcommand{\Magma}{\textsc{Magma}}
\newcommand{\Rev}{\mathrm{Rev}}

\newcommand{\Sy}{\mathsf{S}}
\newcommand{\Sym}{\mathrm{Sym}}
\newcommand{\Z}{\mathbb{Z}}

\makeatletter        
\def\@adminfootnotes{%
  \let\@makefnmark\relax  \let\@thefnmark\relax
  \ifx\@empty\@date\else \@footnotetext{\@setdate}\fi
  \ifx\@empty\@subjclass\else \@footnotetext{\@setsubjclass}\fi
  \ifx\@empty\@keywords\else \@footnotetext{\@setkeywords}\fi
  \ifx\@empty\thankses\else \@footnotetext{%
    \def\par{\let\par\@par}\@setthanks}%
  \fi}\makeatother   

\begin{document}

\newcommand{\lcm}{\mathrm{lcm}}

\hyphenation{}

\title[`Norman involutions' and tensor products]{`Norman involutions' and
  tensor products\\ of unipotent Jordan blocks}
\author{S.\,P. Glasby, Cheryl E. Praeger and Binzhou Xia}

\address[Glasby, Praeger, Xia]{
Centre for Mathematics of Symmetry and Computation\\
University of Western Australia\\
35 Stirling Highway\\
Crawley 6009, Australia.\newline  Email: {\tt Stephen.Glasby@uwa.edu.au; WWW: \href{http://www.maths.uwa.edu.au/~glasby/}{http://www.maths.uwa.edu.au/$\sim$glasby/}}\newline
Email: {\tt Cheryl.Praeger@uwa.edu.au; WWW: \href{http://www.maths.uwa.edu.au/~praeger}{http://www.maths.uwa.edu.au/$\sim$praeger} }
}
\address[Xia]{
Current address: School of Mathematics and Statistics, The University of Melbourne, Parkville, VIC 3010, Australia.\newline Email: {\tt BinzhouX@unimelb.edu.au; WWW: \href{http://ms.unimelb.edu.au/people/profile?id=1612}{https://ms.unimelb.edu.au/people/profile?id=1612}}
}

\date{\today}

\begin{abstract}
  This paper studies the Jordan canonical form (JCF) of the tensor product
  of two unipotent Jordan blocks over a field of
  prime characteristic~$p$. The JCF is characterized by a partition
  $\lambda = \lambda(r,s,p)$ depending on the dimensions $r$, $s$ of the
  Jordan blocks, and on~$p$. Equivalently, we study a permutation
  $\pi = \pi(r,s,p)$ of $\{1,2,\dots,r\}$  introduced by Norman.
  We show that 
  $\pi(r,s,p)$ is an involution involving reversals, or is the identity
  permutation. We prove that the group $G(r,p)$ generated by
  $\pi(r,s,p)$ for all~$s$, ``factors'' as a wreath product 
  corresponding to the factorisation $r=ab$ as a product of its
  $p'$-part~$a$ and $p$-part~$b$: precisely $G(r, p)=\Sy_a\wr \D_b$
  where $\Sy_a$ is a symmetric group of degree~$a$, and $\D_b$ is
  a dihedral group of degree~$b$. We also give simple necessary and sufficient
  conditions for $\pi(r,s,p)$ to be trivial.
\end{abstract}

\maketitle
\centerline{\noindent AMS Subject Classification (2010): 15A69, 15A21, 13C05}

\section{Introduction}

A good knowledge of the decompositions of tensor products of `Jordan blocks'
is key to understanding the actions of $p$-groups of matrices in characteristic $p$. Not surprisingly, the decompositions have applications to algebraic groups~\cites{rL,rL2}.
We explore certain permutations introduced by C. W. Norman~\cite{Norman1995}
in 1995, and groups they generate, to gain new insights into these decompositions.

Let $F$ be a field and let $r,s$ be positive integers with $r\le s$.
The nilpotent Jordan block $N_r$ of degree~$r$ is the $r\times r$ matrix
with 1 in entry $(i,i+1)$, for $1\le i<r$, and zeros elsewhere. Clearly
$N_r\otimes N_s$ is nilpotent of order at most $r$ as
$(N_r\otimes N_s)^r=(N_r)^r\otimes (N_s)^r=0$. It is not hard
to prove that $N_r\otimes N_s$ has order precisely $r$ and
is conjugate via a permutation matrix to
$(s-r+1)N_r\oplus\bigoplus_{i\ge1} 2N_{\mu_i}$, for certain non-negative integers
$\mu_i$. Furthermore, the partition
$(\mu_1,\mu_2,\dots)$ of $\frac12(rs-(s-r+1)r)=\frac12r(r-1)$ is independent of the
field~$F$. It follows that the Jordan canonical form
$\textup{JCF}(N_r\otimes N_s)$ of $N_r\otimes N_s$ equals
$(s-r+1)N_r\oplus\bigoplus_{i\ge1} 2N_{\mu_i}$.
We consider unipotent Jordan block matrices $J_\ell:=I_\ell+N_\ell$, where $I_\ell$
is the $\ell\times \ell$ identity matrix.
By contrast, we have
$\textup{JCF}(J_r\otimes J_s)=\bigoplus_{i\ge1} J_{\lambda_i}$ for some partition
$\lambda=(\lambda_1,\lambda_2,\dots)$ of $rs$ and here both the conjugating matrix
in $\GL_{rs}(F)$, and the partition $\lambda$, depend on~$F$.
The study of the properties of $\lambda$
includes~\cites{B,GPX2,GPX1,II2009,Norman1995,Renaud1979, Srinivasan1964} and
dates back to Aitken~\cite{A} in 1934. The parts of $\lambda$ are dimensions of
indecomposable modules in what is now called the Green ring~\cite{Green1962}.

The above partition $\lambda$ of $rs$ depends only on $r$, $s$, and $p:=\Char(F)$.
The case $p=0$ is completely understood, see \cite{Srinivasan1964}*{Corollary 1}.
We assume $p>0$, and we write
$\lambda$ as $\lambda(r,s,p)$. There exist complicated
algorithms~\cites{II2009,Renaud1979} for computing the parts of
$\lambda(r,s,p)$, and the dependence on the prime $p$ is particularly difficult.
In Section~\ref{sec3}, we consider some underlying combinatorics which is
\emph{independent} of $p$. Roughly we show that various partitions correspond
bijectively to subsets of $\{1,\dots,r-1\}$, and the proper subsets
correspond bijectively to certain involutions (constructed from reversals)
of the symmetric group $\Sy_r$ of degree~$r$.

In Section~\ref{sec:back} we define `Norman permutations' $\pi$ corresponding
to  partitions $\lambda$ in a certain subset of $2^{r-1}$ partitions of $rs$.
These permutations generalise, and were
motivated by, permutations defined by Norman
\cite{Norman1995}*{Equation (13), p.\;353} for partitions corresponding to
$\otimes$-products $J_r\otimes J_s$ over various fields of prime order $p$.  We note that there are far fewer
than $2^{r-1}$ permutations of the form $\pi(r,s,p)$ arising
from $\otimes$-products~\cite{GPX1}*{Table~4}.
Building on~\cite{B}, we give simplified necessary and sufficient
conditions, depending on $(r,s,p)$,
for $\pi(r,s,p)$ to be trivial (Theorem~\ref{newBarry-main}).
Section~\ref{sec1} gives `formulas' (Theorem~\ref{Prop2}(d)) for
the parts of $\lambda(r,s,p)$. This allows us to prove (Theorem~\ref{Prop2})
that the
permutations $\pi(r,s,p)$ correspond, as above, to subsets of $\{1,\dots,r-1\}$.
Section~\ref{s:triv} gives a proof of Theorem~\ref{newBarry-main}.
Finally, in Section~\ref{sec2} we show that the subgroups
$G(r,p)=\langle\pi(r,s,p)\mid r\le s\rangle$ of $\Sy_r$ are
imprimitive. Indeed we prove in Theorem~\ref{T:wreath}
that $G(r,p)=\Sy_a\wr \D_b$ is a wreath product
of a the symmetric group of degree~$a$, and the dihedral group of degree~$b$.
We discovered the essential idea for the proof by using
a \Magma\ computer programs available at~\cite{Gl}. This computational insight
allowed us to replace a number of difficult partial results with a relatively
short proof of Theorem~\ref{T:wreath}.

\section{The main results}\label{sec:back}

Fix integers $r,s$ where $1\le r\le s$.
In general, given $\alpha,\beta\in F$, the Jordan form of
$(\alpha I_r+N_r)\otimes(\beta I_s+N_s)$ is
$\bigoplus_{i\ge1} \alpha\beta I_{\lambda_i}+N_{\lambda_i}$ where
$\lambda=(\lambda_1,\lambda_2,\dots)$ is a partition of $rs$.
When $\alpha\beta=0$, $\lambda$ is easily determined,
see \cite{II2009}*{Proposition~2.1.2}, and if $\alpha\beta\ne0$, then $\lambda$
is independent of the choice of $\alpha,\beta\in F^\times$.
We henceforth take $\alpha=\beta=1$. In this case, $\lambda$
has precisely~$r$ nonzero parts which we write in order
$\lambda_1\ge\cdots\ge\lambda_r>0$. Clearly $\sum_{i=1}^r\lambda_i=rs$.
If $\Char(F)=0$ or  $\Char(F)\ge r+s-1$ then, 
by~\cite{Srinivasan1964}*{Corollary~1}, the parts of $\lambda$ satisfy

\begin{equation}\label{eq:std}
  \lambda_n=r+s+1-2n \qquad \mbox{for $n\in [r]:=\{1,2,\dots,r\}.$}
\end{equation}

Henceforth assume that $p:=\Char(F)>0$. In this case, there exist
algorithms~\cite{Renaud1979}*{Theorem~2},
but no `closed' formula, for the $n$th part $\lambda_n$. Although the dependence
on the prime~$p$ is complicated, there exist duality and
periodicity properties~\cite{GPX1}*{Theorem 4} for $\lambda(r,s,p)$
related to the smallest $p$-power satisfying $r\le p^m$.
In Section~\ref{sec1}, we show that
\begin{equation}\label{eq:formula}
  \lambda_n=r+s+1-2n+L(n)-R(n) \qquad \mbox{for $n\in [r]$}
\end{equation}
depends on `left' and `right' functions $L$ and $R$ (which depend on $p$).
Here $[r]=\{1,\dots,r\}$, for a positive integer $r$.
This `formula'
is based on the algorithm~\cite{II2009}*{Theorem 2.2.9}, and it is used to
show
\begin{align}
  &\lambda_i-\lambda_j\neq j-i\quad&&\text{for}\quad1\leq i<j\leq r,\textup{ and}\label{Eq2}\\
  &1-n\leq \lambda_n -s\leq r-n\quad&&\text{for}\quad n\in[r].\label{Eq3}
\end{align}

To suppress the dependency on the characteristic~$p$, we study partitions
$\lambda$ of $rs$ with $r$ nonzero parts
$\lambda_1\ge\cdots\ge \lambda_r>0$ satisfying~\eqref{Eq2} and~\eqref{Eq3}.
It may be that such a partition $\lambda$ does not equal $\lambda(r,s,p)$ for
\emph{any} prime $p$.
We show that these partitions correspond bijectively to subsets of $[r-1]$,
and that properties~\eqref{Eq2} and~\eqref{Eq3} guarantee that we can define
a bijection of $[r]$.

\begin{definition}\label{D:pi}
Given a partition $\lambda$ of $rs$ with $r$ nonzero parts
$\lambda_1\ge\cdots\ge \lambda_r>0$ satisfying~\eqref{Eq2} and~\eqref{Eq3},
define a permutation
$\pi_\lambda\colon[r]\to[r]$ by
\begin{equation}\label{eq4}
  n^{\pi_\lambda}=(r+1-n)+s-\lambda_n\qquad\textup{for $n\in\{1,\dots,r\}.$}
\end{equation}
We call $\pi_\lambda$ a \emph{Norman permutation} of $[r]$.
When $\lambda=\lambda(r,s,p)$, that is to say, when
$\textup{JCF}(J_r\otimes J_s)=\bigoplus_{i\ge1} J_{\lambda_i}$, we write
$\pi_\lambda=\pi(r,s,p)$.
\end{definition}

We prove in Theorem~\ref{Prop2} that $\pi_\lambda^2=1$, and
when $\pi_\lambda\ne1$ we call $\pi_\lambda$ a \emph{Norman involution}.
The map $\pi(r,s,p)$ was introduced by Norman~\cite{Norman1995}*{p.\,353},
and was shown in \cite{Norman1995}*{Lemma~3} to be a permutation~ of $[r]$.
In particular, it follows from Definition~\ref{D:pi} that
$\pi(r,s,p)$ is the identity if and only if
the corresponding Jordan partition
$\lambda(r,s,p) = (\lambda_1,\dots,\lambda_r)$ satisfies~\eqref{eq:std}.
This condition has been studied in several papers, and in
 \cite{GPX1} the partition $\lambda(r,s,p)$ was called
\emph{standard} if \eqref{eq:std} holds. We use this terminology here.
In 1964, Srinivasan \cite{Srinivasan1964} proved that $\lambda(r,s,p)$
is standard if $p\ge r+s-1$, and we extended her result in
\cite{GPX1}*{Theorem 2}.
Recently, Barry \cite{B} obtained a complete set of necessary
and sufficient conditions on $r, s, p$ for $\lambda(r,s,p)$
to be standard, and hence for $\pi(r,s,p)$ to be trivial.
We found Barry's conditions difficult to verify in some contexts,
and sought a simpler statement of them.
We establish in Theorem~\ref{newBarry-main} an equivalent criterion for
`standard-ness', that is, for \eqref{eq:std} to hold. It
involves only certain congruences on $r, s$ and $p$, and we
find it easier to apply.
We use the following convention: for any integer
$n$ and positive integer $\ell$, denote by $n_{\bmod \ell}$ the unique
integer in the interval $[0,\ell-1]$ which is congruent to $n$ modulo $\ell$.

\begin{definition}\label{stdtriple}
{\rm
Let $1\le r\le s$, let $p$ be a prime, and let $m=\lceil\log_p(r)\rceil$.
The triple $(r,s,p)$ is called \emph{standard} if and only if
the conditions in Table~\ref{tab3} apply, where
\[
  a:=r_{\bmod p^{m-1}},\; b:=s_{\bmod p^{m-1}},\;
  h:=\frac{p^{m-1}-1}{2}, \;
  i:=\left\lfloor\frac{r}{p^{m-1}}\right\rfloor,\;
  j:=\left\lfloor\frac{(s-r+1)_{\bmod p^m}}{p^{m-1}}\right\rfloor.
\]
\begin{center}
\begin{table}[ht!]\label{T:std}
\caption{The conditions for $(r, s, p)$ to be a standard triple.}\label{tab3}
\begin{tabular}{ccl}
  \toprule
$p$	&	$r$	&	Other conditions (if any)	\\ \midrule
any	&	$1$	&	--\\
any 	&	$\quad2\le r\leq p\qquad$&	$(s-r+1)_{\bmod p}\le p+2-2r$\\
$2$	&	$3$	&	$s\equiv 2\pmod{4}$	\\
odd 	&	$r>p$	&	$a-h\in\{0,1\}$, $b-h\in\{0,1\}$ and $2i+j\le p-1$\\ 	\bottomrule
\end{tabular}
\end{table}
\end{center}
}
\end{definition}

\begin{theorem}\label{newBarry-main}
Let $1\le r\le s$ and let $p$ be a prime. Then the following are equivalent:
\begin{enumerate}[{\rm (i)}]
  \item $\lambda(r,s,p)$ is a standard partition;
  \item $\pi(r,s,p)$ is the identity permutation;
  \item $(r,s,p)$ is a standard triple as in \textup{Definition~\ref{stdtriple}}.
\end{enumerate}
\end{theorem}

In particular, this result generalises Srinivasan's theorem in
\cite{Srinivasan1964}*{Corollary 1}:
if $s\geq r\geq 2$ and $p\ge r+s-1$, then the conditions of
line~2 of Table~\ref{T:std} hold, so the triple $(r,s,p)$ is
standard and hence, by Theorem~\ref{newBarry-main}, the partition $\lambda(r,s,p)$
is standard.
A more detailed technical result (Proposition~\ref{newBarry}) is proved in
Section~\ref{s:triv}, which immediately implies
Theorem~\ref{newBarry-main}. In Section~\ref{sec1} we collect
results about Norman permutations.
In Section~\ref{sec2} we study certain groups (Definition~\ref{Def:G})
generated by Norman involutions.

For a set $\Omega$, denote by $\Sym(\Omega)$ the symmetric group
on $\Omega$ and $\Alt(\Omega)$ the alternating group on~$\Omega$.
The groups $\Sym([n])$ and $\Alt([n])$ are also written simply as
$\Sy_n$ and $\A_n$, and these permutation groups are said to
have \emph{degree}~$n$.

\begin{definition}\label{Def:G}
Given an integer $r\geq 2$ and a prime $p$, let $G(r,p)$
be the subgroup $G(r,p):=\langle\,\pi(r,s,p)\mid s\ge r\,\rangle$
of the symmetric group $\Sy_r$ of degree $r$.
\end{definition}

In  Section~\ref{sec2}, we determine precisely the structure of this group.
We denote by $\D_b$ the dihedral group  of degree~$b$; note that 
$|\D_b|=2b$ if $b\geq3$, and $\D_b=\Sy_b$ for $b=1,2$.

\begin{theorem}\label{T:wreath}
Let $p$ be a prime and $r$ a positive integer with 
$p'$-part $a$ and $p$-part $b$. Then
$G(r,p)\cong G(a,p)\wr G(b,p)\cong \Sy_a\wr\D_b.$
\end{theorem}

\section{Subsets of \texorpdfstring{$[r-1]$}{}, \texorpdfstring{$r$}{}-tuples, and permutations of \texorpdfstring{$[r]$}{}}\label{sec3}

Fix a positive integer $r$.
In this section, we define bijections between three sets $\calT, \calE$ and
$\Pi$. The first set  $\calT$ is the set of subsets of $[r-1] = \{ 1,\dots,r-1\}$, so that
$|\calT|=2^{r-1}$. The second set $\calE$ comprises all
$r$-tuples $\eps=(\eps_1,\eps_2,\dots,\eps_r)\in\mathbb{Z}^r$
with $\eps_1\ge \eps_2\ge\cdots\ge \eps_r$ satisfying
\begin{align}
  &\eps_i-\eps_j\neq j-i\quad\text{for}\quad1\leq i<j\leq r,\textup{ and}\label{eq2}\\
  &1-n\leq \eps_n\leq r-n\quad\text{for}\quad1\leq n\leq r.\label{eq3}
\end{align}
We define $\Phi_{\calE\calT}:\calE \rightarrow \calT$ as follows. For $\eps\in\calE$ (so~\eqref{eq2}
and~\eqref{eq3} hold),
define a subset $T_\eps$ of $[r-1]$ by
\begin{equation}\label{eqTe}
  T_\eps=\eps^{\Phi_{\calE\calT}}=\{t\in[r-1]\mid \eps_t>\eps_{t+1}\}.
\end{equation}
Clearly $\Phi_{\calE\calT}$ is well defined.
For example, the constant $r$-tuple $\eps = (0,\dots,0)$ lies in $\calE$ and
$T_\eps = \emptyset$. We show in Theorem~\ref{prop3} that
$\Phi_{\calE\calT}$ is a bijection.
The third set $\Pi$ is a subset of~$\Sy_r$. Before defining it,
we first define `reversal
permutations' of $[r]$.

\begin{definition}\label{defRev}
For $1\le i\le j\le r$, the \emph{reversal permutation} $\Rev(i,j)$
of $\Sy_r$ is the identity if $i=j$, and otherwise $\Rev(i,j)$ 
sends the sequence $(i,i+1,\dots,j)$ to $(j,j-1,\dots,i)$
and fixes $[r]\setminus\{i,i+1,\dots,j\}$ pointwise.
\end{definition}

Note that $n^{\Rev(i,j)}$ equals $j + i - n$ for $i\leq n\leq j$, and
equals~$n$
otherwise. Thus $\Rev(i,j)$ has disjoint cycle decomposition
$\prod_{k=0}^{\lfloor(j-i-1)/2\rfloor}(i+k,j-k)$. Further,
$\Rev(i,j)$ is an involution if $i<j$, and is the identity if $i=j$.
If $j-i=2k$ is even, then the final term $(i+k,j-k)$ in the product is
a 1-cycle which we omit, e.g., $\Rev(1,5)=(1,5)(2,4)$.

Given a subset $T=\{t_1,\dots,t_{|T|}\}$ of
$[r-1]$, we assume $t_1<\cdots<t_{|T|}$, and we define $t_0:=0$, $t_{|T|+1}:=r$,
and write
\begin{equation}\label{eq6}
 \pi_T=\prod_{i=0}^{|T|}\Rev(t_i+1,t_{i+1}) \in\Sy_r.
\end{equation}
For example, $\pi_{[r-1]} = 1$ and $\pi_\emptyset  = \Rev(1,r)$.
We define the third set as $\Pi:=\left\{\pi_T\mid T\in\calT\right\}$,
and we define the function
\begin{equation}\label{eq6b}
  \Phi_{\calT\Pi}\colon \calT\to\Pi\quad\textup{by}\quad
  T^{\Phi_{\calT\Pi}}=\pi_T.
\end{equation}
By definition $\Phi_{\calT\Pi}$ is surjective. Further, distinct subsets $T$ and
$T'$ of $[r-1]$
clearly correspond to distinct permutations $\pi_T$ and $\pi_{T'}$. Thus
$\Phi_{\calT\Pi}$ is also injective, and hence bijective.
The set $\calE$ of $r$-tuples is somewhat mysterious.
For $\eps=(\eps_1,\dots,\eps_r)\in\calE$, we~define 
\begin{equation}\label{eq7}
  \pi_\eps\colon [r]\to[r]\quad\textup{by}\quad
  n^{\pi_\eps}=r+1-n-\eps_n\quad\text{for}\quad1\leq n\leq r.
\end{equation}
Condition~\eqref{eq3} implies that $[r]^{\pi_\eps}\subseteq[r]$, and condition~\eqref{eq2}
means that $\pi_\eps$ is injective, and hence bijective. Thus $\pi_\eps\in\Sy_r$.
We show in Lemma~\ref{lem:calF} that $\pi_\eps\in\Pi$, and hence
that 
\begin{equation}\label{eq7b}
\mbox{$\eps^{\Phi_{\calE\Pi}}=\pi_\eps$ \quad defines a map \quad $\calE\to\Pi$.}
\end{equation}
Anticipating our proofs that the maps we have defined are bijections, and writing, for example,
   $\Phi_{\Pi\calE}=\Phi_{\calE\Pi}^{-1}$ etc., we illustrate these maps in Figure~\ref{F1}.

\begin{figure}[ht!]
  \caption{Bijections between $\calE$, $\calT$, and $\Pi$ where
   $\Phi_{\Pi\calE}=\Phi_{\calE\Pi}^{-1}$ etc.}
\vskip2mm
  \label{F1}
\begin{tikzpicture}[xscale=0.8, yscale=0.8]
  \draw[line width=1pt,fill][->] (0,0) node[left,black] {$\calE$}--(0.5,0.866);
  \draw[line width=1pt,fill]     (0.5,0.866) node[left,black] {$\Phi_{\calE\calT}$} -- (1,1.732);
  \draw[line width=1pt,fill][->] (1,1.732) node[left,black] {$\calT$} -- (1.5,0.866);
  \draw[line width=1pt,fill]     (1.5,0.866) node[right,black] {$\Phi_{\calT\Pi}$}-- (2,0);
  \draw[line width=1pt,fill][->] (2,0) node[right,black] {$\Pi$} -- (1,0);
  \draw[line width=1pt,fill]     (1,0) node[below,black] {$\Phi_{\Pi\calE}$} -- (0,0);
  \draw[fill] (0,0) circle [radius=0.1];
  \draw[fill] (2,0) circle [radius=0.1];
  \draw[fill] (1,1.732) circle [radius=0.1];
\end{tikzpicture}
\hspace{30mm}
\begin{tikzpicture}[xscale=0.8, yscale=0.8]
  \draw[line width=1pt,fill] (0,0) node[left,black] {$\calE$}--(0.5,0.866);
  \draw[line width=1pt,fill][<-]     (0.5,0.866) node[left,black] {$\Phi_{\calT\calE}$} -- (1,1.732);
  \draw[line width=1pt,fill] (1,1.732) node[left,black] {$\calT$} -- (1.5,0.866);
  \draw[line width=1pt,fill][<-]     (1.5,0.866) node[right,black] {$\Phi_{\Pi\calT}$}-- (2,0);
  \draw[line width=1pt,fill] (2,0) node[right,black] {$\Pi$} -- (1,0);
  \draw[line width=1pt,fill][<-]     (1,0) node[below,black] {$\Phi_{\calE\Pi}$} -- (0,0);
  \draw[fill] (0,0) circle [radius=0.1];
  \draw[fill] (2,0) circle [radius=0.1];
  \draw[fill] (1,1.732) circle [radius=0.1];
\end{tikzpicture}
\end{figure}

\noindent
Note that each $T\in\calT$ defines a partition
$\{\Delta_0,\dots,\Delta_{|T|}\}$ of $[r]$ where
\begin{equation}\label{eq8}
  \Delta_i=[t_i+1,t_{i+1}]\cap\mathbb{Z}=\{t_i+1,t_i+2,\dots,t_{i+1}\}
  \qquad\text{for}\quad i=0,1,\dots,|T|.
\end{equation}
By definition, the permutation $\pi_T=T^{\Phi_{\calT\Pi}}$ reverses the elements
in each  interval $\Delta_i$.

\begin{lemma}\label{lem:calF}
For $\eps\in\calE$, the function $\pi_\eps$ defined by~\eqref{eq7} lies in $\Pi$.
\end{lemma}

\begin{proof}
Let $\eps\in\calE$, and let $\pi_\eps\in\Sy_r$ as defined in~\eqref{eq7}, and
$T=T_\eps$ as defined in~\eqref{eqTe}.
If $T=\{t_1,\dots,t_{|T|}\}$, then the function $n\mapsto \eps_n$ is constant
on the sets $\Delta_i$ given by~\eqref{eq8}. Moreover, different $\eps$-values
are taken on different~$\Delta_i$.

For subsets $\Delta$ and $\Delta'$ of $\mathbb{Z}$, write
$\Delta < \Delta'$ if $\delta < \delta'$ for each  $\delta\in\Delta$
and $\delta'\in\Delta'$.
Then the partition $\{\Delta_0,\Delta_1,\dots,\Delta_{|T|}\}$ of~$[r]$
satisfies $\Delta_0<\Delta_1<\cdots<\Delta_{|T|}$.
For $0\leq i\leq|T|$, the bijection
$\pi_\eps$ induces a bijection from $\Delta_i$ to $\Delta_i^{\pi_\eps}$. Thus
$|\Delta_i^{\pi_\eps}|=|\Delta_i|=t_{i+1}-t_i\ge1$. Clearly
$\{\Delta_0^{\pi_\eps}, \Delta_1^{\pi_\eps}, \dots, \Delta_{|T|}^{\pi_\eps}\}$ is a partition of $[r]$.
For $0\leq i\leq|T|$, the function $n\mapsto \eps_n$ is constant for $n\in\Delta_i$, and hence we deduce from~\eqref{eq7} that
\[
  \Delta_i^{\pi_\eps}=[t_{i+1}^{\pi_\eps}, (t_i+1)^{\pi_\eps} ]\cap\mathbb{Z}
  =\{t_{i+1}^{\pi_\eps}, \dots, (t_i+2)^{\pi_\eps}, (t_i+1)^{\pi_\eps}\}
\]
is a set of consecutive integers. Moreover, since $\eps_{t_{i+1}+1}\leq \eps_{t_{i+1}}-1$, we obtain
\begin{align*}
  (t_{i+1}+1)^{\pi_\eps} &= r+1-\eps_{t_{i+1}+1}-(t_{i+1}+1)\\
  &\geq r+1-(\eps_{t_{i+1}}-1)-(t_{i+1}+1) = t_{i+1}^{\pi_\eps}.
\end{align*}
As $\Delta_i^{\pi_\eps}$ and $\Delta_{i+1}^{\pi_\eps}$ are disjoint sets of consecutive integers,
we deduce that $\Delta_i^{\pi_\eps}< \Delta_{i+1}^{\pi_\eps}$ for each $ i\leq|T|$.
Thus $\{\Delta_0^{\pi_\eps}, \Delta_1^{\pi_\eps}, \dots, \Delta_{|T|}^{\pi_\eps}\}$ 
is a partition of $[r]$ with  $\Delta_0^{\pi_\eps} < \Delta_1^{\pi_\eps} < \cdots 
< \Delta_{|T|}^{\pi_\eps}$  and
$|\Delta_i^{\pi_\eps}|=|\Delta_i|$ for $i=0,1,\dots,|T|$. We conclude that
$\Delta_i^{\pi_\eps}=\Delta_i$ for $i=0,1,\dots,|T|$. This together with~\eqref{eq7}
implies that $\pi_\eps\vert_{\Delta_i}$ equals $\Rev(t_i+1,t_{i+1})$. Hence $\pi_\eps$ equals $T^{\Phi_{\calT\Pi}}$ by~\eqref{eq6},
and therefore $\pi_\eps\in\Pi$ as desired.
\end{proof}

\begin{theorem}\label{prop3}
The maps $\Phi_{\calE\calT}$, $\Phi_{\calT\Pi}$ and $\Phi_{\calE\Pi}$  defined
in $\eqref{eqTe}, \eqref{eq6b}$ and $\eqref{eq7b}$, are all bijections and
$\Phi_{\calE\calT}\circ \Phi_{\calT\Pi}\circ \Phi_{\Pi\calE}=1$ where
$\Phi_{\Pi\calE}=\Phi_{\calE\Pi}^{-1}$. Furthermore, using equations \eqref{eq6} and \eqref{eq7}, we have
\begin{enumerate}[{\rm (a)}]
  \item $\pi_T^2=1$ for all $T\in\calT$, and $\pi_T=1$ if and
  only if $T=[r-1]$; and
  \item  $\pi_\eps^2=1$ for all $\eps\in\calE$, and $\pi_\eps=1$ if and
  only if $\eps_n=r+1-2n$ for each $n\in[r]$.
\end{enumerate}
\end{theorem}

\begin{proof}
By the discussion after \eqref{eq6b}, $\Phi_{\calT\Pi}$ is a bijection.
Also, the map $\Phi_{\calE\Pi}\colon\calE\to\Pi$ with $\pi_\eps=\eps^{\Phi_{\calE\Pi}}$
as in~\eqref{eq7} is clearly injective.
To show that $\Phi_{\calE\Pi}$ is surjective, take an arbitrary
$\pi\in\Pi$ and define an $r$-tuple $\eps_\pi=(\eps_1,\dots,\eps_r)$ by
\begin{equation}\label{eq10}
  \eps_n=r+1-n- n^\pi \qquad\textup{for $n\in [r]$.}
\end{equation}
We claim that $\eps_\pi\in \calE$. Since $\Phi_{\calT\Pi}$ is a bijection, we have
$\pi=T^{\Phi_{\calT\Pi}}$ for some $T\in\calT$. Define
the corresponding  partition $\{\Delta_0,\Delta_1,\dots,\Delta_{|T|}\}$ of $[r]$
by~\eqref{eq8}, and recall (as noted after~\eqref{eq8}) that
$\pi\mid_{\Delta_i} = \Rev(t_i+1,t_{i+1})$ for each $i$.

If $n,n+1\in\Delta_i$, then $(n+1)^\pi=n^\pi-1$ and it follows from \eqref{eq10}
that $\eps_n=\eps_{n+1}$.
If $n, n+1$ are not in the same part, then for some $i$,
$n=t_i\in\Delta_{i-1}$ and $n+1=t_i+1\in\Delta_{i}$.
As we noted above we then have $t_i^\pi=t_{i-1}+1$
and $(t_i+1)^\pi = t_{i+1}$ and so, by~\eqref{eq10},
$\eps_n=r+1-t_i-t_{i-1}-1$, which is strictly larger than
$\eps_{n+1}=r+1-t_i-t_{i+1}-1$. (Indeed, $\eps_n\ge \eps_{n+1}+2$.)
This shows that $\eps_1\ge\cdots\ge \eps_r$ holds.
Since $\pi$ is injective,~\eqref{eq2} holds, and since
$[r]^\pi\subseteq[r]$,~\eqref{eq3} holds.
Thus $\eps\in \calE$ as claimed. Moreover it follows from \eqref{eq7}
and \eqref{eq10} that $\eps_\pi^{\Phi_{\calE\Pi}} = \pi$.  Hence $\Phi_{\calE\Pi}$
is bijective. Therefore the map $\Phi_{\Pi\calE}\colon\Pi\to\calE$ defined
by $\pi\mapsto \eps_\pi$ as in~\eqref{eq10} is well-defined, and it equals $\Phi_{\calE\Pi}^{-1}$.

Next we consider the assertions (a) and (b).  
Let $T=\{t_1,\dots,t_{|T|}\}$ be a subset of
$[r-1]$, with $t_1<\cdots<t_{|T|}$.
Then $\pi_T=T^{\Phi_{\calT\Pi}}$, defined by \eqref{eq6} with $t_0:=0$ and $t_{|T|+1}:=r$,  is a product of the reversal permutations
$\Rev(t_i+1,t_{i+1})$ for each $i$. Thus  $\pi_T^2=1$.
Moreover $\pi_T=1$
if and only if  $t_{i+1}=t_i+1$ for each $i$, and this holds if and only if $T=[r-1]$.
Similarly, by \eqref{eq7}, $\pi_\eps = 1$ if and only if
$\eps_n=r+1-2n$ for each $n$. Thus parts (a) and (b) are proved.
 Finally, it now follows that
$\eps^{\Phi_{\calE\calT}\circ \Phi_{\calT\Pi}\circ \Phi_{\Pi\calE}}=\eps$ for $\eps\in\calE$.
Hence, $\Phi_{\calE\calT}\circ \Phi_{\calT\Pi}\circ \Phi_{\Pi\calE}=1$, and in particular
$\Phi_{\calE\calT}$ is a bijection.
This completes the proof.
\end{proof}

\begin{remark}
Fix a subset $T=\{t_1,\dots,t_{|T|}\}$ of $[r-1]$, with $t_1<\cdots<t_{|T|}$.
Set $t_0:=0$ and $t_{|T|+1}:=r$. It was shown in the proof
of Theorem~\ref{prop3}
that $\eps=T^{\Phi_{\calT\calE}}$ equals $(\eps_1,\dots,\eps_r)$ where
$\eps_n=r+1-t_i-t_{i+1}$ for $n\in\Delta_i=[t_i+1,t_{i+1}]\cap\Z$.
Thus the $\eps$-values are constant on $\Delta_i$ and on $\Delta_{i+1}$,
and the value on $\Delta_{i+1}$ is at least $2$ less than the value
on $\Delta_{i}$.
\end{remark}

\section{Norman permutations}\label{sec1}

We use the results of the previous section
to give formulas for $\pi(r,s,p)$ and to show that $\pi(r,s,p)^2=1$.


For $n\in[r]$, define the integer $D_n(r,s)$ to be the determinant of the
$n\times n$ matrix with $(i,j)$-entry $\binom{s+r-2n}{s-n+i-j}$,
where $0\le i,j\le n-1$. It was shown in \cite{Robert1994} that
\begin{equation}\label{eq1}
  D_n(r,s)
  =\prod\limits_{i=0}^{n-1}\frac{\binom{s+r-2n+i}{s-n}}{\binom{s-n+i}{s-n}},
  \qquad\textup{for $n\in[r]$.}
\end{equation}
Set $D_0(r,s)=1$, and note that $D_r(r,s)=1$. For $n=0,1,\dots,r$, define
\begin{equation*}
\delta_n=\delta_n(r,s,p)=
\begin{cases}
0\quad&\text{if $D_n(r,s)\equiv0\pmod{p}$,}\\
1\quad&\text{if $D_n(r,s)\not\equiv0\pmod{p}$.}
\end{cases}
\end{equation*}
For $n\in[r]$, let $L(n)$ be the smallest positive integer such
that $\delta_{n-L(n)}=1$ and $R(n)$ be the least nonnegative integer
such that $\delta_{n+R(n)}=1$. Note that $L(n)$ and $R(n)$ are well defined since
$\delta_0=1$ and $\delta_r=1$, and in particular $L(1)=1$.
Moreover, $1\le L(n)\le n$ and $0\le R(n)\le r-n$ hold, for $n\in[r]$.
Table~\ref{tab2} illustrates how these functions determine the values of
$\delta_k$ as $k$ varies.

\begin{center}
\begin{table}[ht!]
\caption{The functions $L$ and $R$ determine when $\delta_k$ is 0 or 1.}\label{tab2}
\begin{tabular}{cccccccccc} \toprule
$k$&$n-L(n)$&$n-L(n)+1$&$\cdots$&$n-1$&$n$&$n+1$&$\cdots$&$n+R(n)-1$&$n+R(n)$\\ \midrule
$\delta_k$&1&0&$\cdots$&0&$\delta_n$&$0$&$\cdots$&0&1\\ \bottomrule
\end{tabular}
\end{table}
\end{center}

For a Jordan partition $\lambda = \lambda(r,s,p) = (\lambda_1,\dots,\lambda_r)$,
define
\begin{equation}\label{eqelambda}
\eps = \eps(r,s,p) = (\lambda_1-s, \dots, \lambda_r - s).
\end{equation}
This $r$-tuple is the \emph{deviation vector} for $\lambda$
defined and studied in \cite{GPX1}*{Definition 1(c)}. Let 
\begin{equation}\label{eq5}
T=\{i\in[r-1]\mid\delta_i=1\}.
\end{equation}
Then the map $\delta: i\mapsto \delta_i$ on $[r]\cup\{0\}$ is the characteristic function for $T\cup\{0,r\}$.

\begin{theorem}\label{Prop2}
Given integers $r$ and $s$ with $1\leq r\leq s$ and a prime number $p$, let
$\lambda=\lambda(r,s,p)=(\lambda_1,\dots,\lambda_r)$, $\eps$ be as in $\eqref{eqelambda}$,
and $T$ be as in $\eqref{eq5}$. Then
\begin{enumerate}[{\rm (a)}]
  \item  $\eps=T^{\Phi_{\calT\calE}}\in\calE$;
  \item  if $T=\{t_1,\dots,t_{|T|}\}$ and $0=t_0<t_1<\cdots<t_{|T|}<t_{|T|+1}=r$,
then
  \[
    \pi(r,s,p)=T^{\Phi_{\calT\Pi}}=\prod_{i=0}^{|T|}\Rev(t_i+1,t_{i+1});
  \]
  \item  $n^{\pi(r,s,p)}=n+1-L(n)+R(n)$, for each $n\in[r]$;
  \item  $\lambda_n=r+s-2n+L(n)-R(n)$, for each $n\in[r]$.
\end{enumerate}
\end{theorem}

\begin{proof}
Let $T=\{t_1,\dots,t_{|T|}\}$ with $0=t_0<t_1<\cdots<t_{|T|}<t_{|T|+1}=r$. 
Write $f=T^{\Phi_{\calT\Pi}}$ and $e=f^{\Phi_{\Pi\calE}}=(e_1,e_2,\dots,e_r)$,
so that $e=T^{\Phi_{\calT\calE}}$, by Theorem~\ref{prop3}.  We shall prove that $e=\eps$.
By \eqref{eq6b},
\begin{equation}\label{eq3.2}
  f=\prod_{i=0}^{|T|}\Rev(t_i+1,t_{i+1}).
\end{equation}
Let $n\in[r]$. If $\delta_n=1$, then by \eqref{eq5}, $n=t_{i+1}$ for some $i$ with $0\leq i\leq|T|$, 
and so by~\eqref{eq3.2} and the definition of $L$,
\begin{equation}\label{eq3.5}
  n^f=t_i+1=n+1-(t_{i+1}-t_i)=n+1-L(n).
\end{equation}
If $\delta_n=0$, then $n\neq t_{i+1}$ for all $i$ with $0\leq i\leq|T|$, and so by~\eqref{eq3.2}, 
$n^f=(n+1)^f+1$. Therefore, by~\eqref{eq10} and~\eqref{eq3.5},
$e_{n+1} = r+1 - (n+1) - (n+1)^f= r-n-(n+1)^f$, and 
\begin{align*}
e_n=r+1-n-n^f
  &=\begin{cases}
  r-2n+L(n)\quad&\textup{if $\delta_n=1$,}\\
  r-n-(n+1)^f&\textup{if $\delta_n=0$,}\\
  \end{cases}\\
  &=\begin{cases}
  r-2n+L(n)\hskip6mm&\textup{if $\delta_n=1$,}\\
  e_{n+1}&\textup{if $\delta_n=0$}.\\
  \end{cases}
\end{align*}
According to \cite{II2009}*{Theorem 2.2.9} or \cite{GPX1}*{Theorem~15}, we have
\[
\lambda_n=
  \begin{cases}
  r+s-2n+L(n)\quad&\textup{if $\delta_n=1$,}\\
  \lambda_{n+1}&\textup{if $\delta_n=0$}.\\
  \end{cases}\\
\]
By \eqref{eqelambda}, $\eps_n=\lambda_n-s$ for all $n$. Thus if $\delta_n=1$ then $\eps_n=r-2n+L(n)=e_n$.
Suppose that   $\delta_n=0$. Then $\eps_n= \lambda_{n+1}-s = \eps_{n+1}$ and $e_n=e_{n+1}$; 
if $\delta_{n+1}=1$ then, as we just showed, $\eps_{n+1}=e_{n+1}$ and so $\eps_n=e_n$, while if 
$\delta_{n+1}=0$, then $\eps_n=\eps_{n+1}=\eps_{n+2}$ and $e_n=e_{n+1}=e_{n+2}$; recursively we conclude  
(recalling that $\delta_r=1$) that also $\eps_n=e_n$ if $\delta_{n+1}=0$. Thus $\eps=e$, and part (a) is proved.

By \eqref{eq3.2}, in order to prove part (b) it remains for us to show that $f$ is equal to $\pi_\lambda=\pi(r,s,p)$. 
Let $n\in [r]$. By Definition~\ref{D:pi}, $n^{\pi_\lambda} = r+1-n+s-\lambda_n$, which by 
\eqref{eqelambda} is equal to $r+1-n-\eps_n$, and this in turn, by \eqref{eq7}, is equal to $n^{\pi_\eps}$. 
Thus $\pi_\lambda = \pi_\eps$, and hence by \eqref{eq7b}, $\pi_\lambda = \eps^{\Phi_{\calE\Pi}}$.
Applying   part (a) and then Theorem~\ref{prop3},  we obtain 
$\eps^{\Phi_{\calE\Pi}}=T^{\Phi_{\calT\calE}\Phi_{\calE\Pi}}=T^{\Phi_{\calT\Pi}}=f$.
Thus part (b) is proved.

Suppose $n\in\Delta_i:=[t_i+1,t_{i+1}]\cap\Z$ for some $i$ with $0\leq i\leq|T|$. 
Then $t_i=n-L(n)$ and $t_{i+1}=n+R(n)$. By part~(b), the restriction of $\pi(r,s,p)$ 
to $\Delta_i$ is $\Rev(t_i+1,t_{i+1})$. Hence 
\begin{align*}
n^{\pi(r,s,p)}=n^{\Rev(t_i+1,t_{i+1})}
&=t_i+1+t_{i+1}-n\\
&=(n-L(n))+1+(n+R(n))-n=n+1-L(n)+R(n),
\end{align*}
proving part~(c).

Finally, by Definition~\ref{D:pi} and part~(c) we have
\begin{align*}
\lambda_n&=r+1-n+s-n^{\pi(r,s,p)}\\
&=r+1-n+s-(n+1-L(n)+R(n))=r+s-2n+L(n)-R(n).
\end{align*}
This proves part~(d).
\end{proof}

\begin{corollary}\label{cor2}
For integers $r$ and $s$ with $1\leq r\leq s$ and a prime number $p$, the Norman permutation $\pi(r,s,p)$
satisfies  $\pi(r,s,p)^2=1$.
\end{corollary}

\begin{proof}
Theorem~\ref{Prop2} shows that $\pi(r,s,p)$ is a product
of reversals related to the subset $T=\{i\in[r-1]\mid \delta_i=1\}$,
and in particular, $\pi(r,s,p)^2=1$.
\end{proof}

\section{Translating from Jordan partitions \texorpdfstring{$\lambda$}{} to Norman permutations \texorpdfstring{$\pi$}{}}

The partition $\lambda(r,s,p)$ and the permutation $\pi(r,s,p)$ determine each other by \eqref{eq4}. In this section we list some results for $\pi(r,s,p)$
which all have a simpler form than the corresponding results for $\lambda(r,s,p)$.

\subsection{Recursive results}
First we prove two results which provide links between Norman involutions for certain different parameters.

\begin{proposition}\label{lem6}
Let $r, s, s', m$ be positive integers such that $r\leq s$, and let $p$ be a prime.
Then the following statements hold.
\begin{enumerate}[{\rm (a)}]
  \item If  $s\le p\le r+s-2$, then $\pi(r,s,p)=\Rev(1,r+s-p)$.
  \item If $r\le p^m$, then $\pi(r,s,p)=\pi(r,s+kp^m,p)$ for each $k\ge0$.
  \item If $r\le\min\{s,s',p^m\}$ and $s+s'\equiv 0\pmod{p^m}$, then
   $ \pi(r,s',p)=\pi(r,s,p)^{\Rev(1,r)}.$
\end{enumerate}
\end{proposition}

\begin{proof}
(a) Suppose that $s\le p\le r+s-2$. By~\cite{Renaud1979}*{Theorem~1},
\[
  \lambda(r,s,p)=
  (\underbrace{\,p,\dots,p}_{r+s-p},2p-r-s-1,2p-r-s-3,\dots,s-r+1).
\]
Thus $\lambda_n=p$ if $1\leq n\leq r+s-p$, and $\lambda_n = r+s+1-2n$
if $r+s-p+1\leq n\leq r$.  Hence, by~\eqref{eq4},
\[
  n^{\pi(r,s,p)}=
  \begin{cases}
    r+s-p+1-n\quad&\textup{if $1\leq n\leq r+s-p$},\\
    n\quad&\textup{if $r+s-p+1\leq n\leq r$}.
  \end{cases}
\]
Therefore $\pi(r,s,p)=\Rev(1,r+s-p)$ as claimed.

(b) Let $k\ge0$ and let $\lambda_1'\ge\dots\ge\lambda_r'$
be the parts of $\lambda(r,s+kp^m,p)$. Then by \cite{GPX1}*{Theorem~4(a)},
$\lambda_n-s=\lambda_n'-(s+kp^m)$ for $n=1,\dots,r$. This implies by \eqref{eq4}
that $\pi(r,s,p)=\pi(r,s+kp^m,p)$.

(c) Write $\lambda(r,s,p)=(\lambda_1,\dots,\lambda_r)$ and
$\lambda(r,s',p)=(\lambda'_1,\dots,\lambda'_r)$, where
$\lambda_1\geq \cdots\geq \lambda_r$ and $\lambda_1'\geq \dots\geq\lambda_r'$.
For $n\in[r]$ write
$\eps_n=\lambda_n-s$ and $\eps'_n=\lambda'_n-s'$.
Then $n^{\pi(r,s,p)}$ equals $r-n+1-\eps_n$ by \eqref{eq4}, and it follows
from~\cite{GPX1}*{Theorem~4(b)} that $-\eps_{r-n+1}=\eps'_n$. Therefore,
using \eqref{eq4}, \eqref{eqelambda}, and Definition~\ref{defRev},
\begin{align*}
  n^{\Rev(1,r)\pi(r,s,p)\Rev(1,r)}&=(r-n+1)^{\pi(r,s,p)\Rev(1,r)}\\
  &=((r-(r-n+1)+1)-\eps_{r-n+1})^{\Rev(1,r)}\\
  &=(n+\eps'_n)^{\Rev(1,r)}=(r-(n+\eps'_n)+1)=n^{\pi(r,s',p)}.\qedhere\\
\end{align*}
\end{proof}

For $m< n$, we view $\Sy_m$ as a subgroup of $\Sy_n$ fixing the set
$\{m+1,\dots,n\}$ elementwise. This convention is employed in the following lemma to 
obtain a reduction formula for $\pi(r,s,p)$.  We also use the deviation vector
$ \eps(r,s,p)$ of a  Jordan partition $\lambda(r,s,p)$ defined in \eqref{eqelambda}.

\begin{proposition}\label{T:s1}
Let $p$ be a prime, and let $r, s_0, s_1$ be integers satisfying $1\le s_0<p$ and $1\le s_1< r < p^m$.
\begin{enumerate}[{\rm (a)}]
  \item If $s=s_0p^m+s_1$, then $\pi(r,s,p)=\pi(s_1,r,p)\circ\Rev(s_1+1,r)$.
  \item If $s'=(s_0+1)p^m-s_1$, then $\pi(r,s',p)=\Rev(1,r-s_1)\circ\pi(s_1,r,p)^{\Rev(1,r)}$.
\end{enumerate}
\end{proposition}

\begin{proof}
(a) The Jordan partition $\lambda(r,s,p)=(\lambda_1,\dots,\lambda_r)$ gives rise
to an equation in the Green ring~\cite{Green1962}, namely
$V_r\otimes V_s=V_{\lambda_1}\oplus\cdots\oplus V_{\lambda_r}$ where $V_k$ denotes
an indecomposable $J_k$-module of dimension~$k$ over the field of order $p$.
Although the precise details are not
significant here, we note that translating between a partition
and a module decomposition is straightforward. The notation
$rV_s$ means $\bigoplus_{i=1}^r V_s$, and it corresponds to the
partition $(s,\dots,s)$ of $rs$ with $r$ parts.
Renaud~\cite{Renaud1979}*{Theorem~2} gave a very complicated recurrence algorithm for decomposing
$V_r\otimes V_s$.  We refer to this algorithm with the notation used
in~\cite{GPX2}*{Proposition~6}.
Set $r_0=0$ and $r_1=r$ so that $r=r_0 p^m+r_1$. Also set $s=s_0 p^m+s_1$,
with $s_0, s_1$ as in the statement. Then $1\le r < p^m < s<p^{m+1}$.  Thus the
hypotheses of~\cite{GPX2}*{Proposition~6} hold with $n=m$, and we have also
$r_0+s_0<p$, so that the parameters $c, d_1, d_2$ of
\cite{GPX2}*{Proposition~6} satisfy $c=d_1=d_2=0$. To aid the exposition we change 
our usage of the $\lambda_i$ and suppose that
$V_{s_1}\otimes V_r=V_r\otimes V_{s_1}=V_{\lambda_1}\oplus\cdots\oplus V_{\lambda_{s_1}}$,
where $\lambda_1\geq \lambda_2\geq \dots\geq \lambda_{s_1}>0$,
so $\lambda(s_1, r, p)=(\lambda_1, \lambda_2, \dots, \lambda_{s_1})$.
Then~\cite{GPX2}*{Proposition~6} gives
\[
  V_r\otimes V_s=(r-s_1)V_{s_0p^m}\oplus\bigoplus_{j=1}^{s_1} V_{s_0p^m+\lambda_j},
\]
in other words,
$\lambda(r,s,p)=(s_0p^m+\lambda_1,\dots,s_0p^m+\lambda_{s_1},s_0p^m,\dots,s_0p^m)$,
where the part $s_0p^m$ has multiplicity $r-s_1>0$.
Therefore the Jordan permutation for this partition satisfies, by \eqref{eq4},
\[
  n^{\pi(r,s,p)}=\begin{cases}
               r+1-n+s_1-\lambda_n&\quad\textup{for $1\le n\le s_1$,}\\
               r+1-n+s_1&\quad\textup{for $s_1< n\le r$.}\end{cases}
\]
Recall that $\pi(s_1, r, p)$ is, by definition, a permutation in $\Sy_{s_1}$ and we identify it with the permutation 
of $[r]$ which acts as $\pi(s_1, r, p)$ on $[s_1]$ and fixes the remaining points. Then
referring again to  \eqref{eq4} for the partition $\lambda(s_1,r,p)$, and to Definition~\ref{defRev} for
$\Rev(s_1+1,r)$, we see that
$\pi(r,s_0p^m+s_1,p)=\pi(s_1,r,p)\circ\Rev(s_1+1,r)$.

(b) Let  $s'=(s_0+1)p^m-s_1$, $s=s_0p^m=s_1$, and note that $s+s'\equiv 0\pmod{p^m}$. Since $s_0>0$, we have
$r\leq\min\{s,s',p^m\}$, and it follows from Proposition~\ref{lem6}~(c)
that $\pi(r,s',p) = \pi(r, s, p)^{\Rev(1,r)}$.  By part (a), this is equal to
$(\pi(s_1,r,p)\circ\Rev(s_1+1,r))^{\Rev(1,r)}$.  A straightforward computation shows that
$\Rev(s_1+1,r)^{\Rev(1,r)} = \Rev(1, r-s_1)$.  Note that, since the support (set of moved points) of
$\pi(s_1,r,p)$ is contained in $[s_1]$, the support of $\pi(s_1,r,p)^{\Rev(1,r)}$ is therefore contained in 
$[s_1]^{\Rev(1,r)}=\{r-s_1+1, \dots,r\}$ which is disjoint from the support of $\Rev(1,r-s)$.  
Part (b) follows on noting that the two factors commute.
\end{proof}

Next we use the expression for $\pi(r,s,p)$ given by Theorem~\ref{Prop2} to deduce a similar expression for
$\pi(p^\ell r,p^\ell s,p)$ for arbitrary $\ell$. This may be viewed as a 
 \emph{`$p^\ell$-multiple proposition'}.

\begin{proposition}\label{prop4}
Suppose (as in  Theorem~\textup{\ref{Prop2}}) that $\pi(r,s,p)=\prod_{i=1}^k\Rev(t_{i-1}+1,t_i)$ with integers
$$
0=t_0<t_1<\dots<t_{k-1}<t_k=r,
$$
and let $\ell$ be a positive integer.
Then $\pi(p^\ell r,p^\ell s,p)=\prod_{i=1}^k\Rev(p^\ell t_{i-1}+1,p^\ell t_i)$.
\end{proposition}

\begin{proof}
Let $\lambda(r,s,p)=(\lambda_1,\dots,\lambda_r)$ and $\lambda(p^\ell r,p^\ell s,p)=(\lambda'_1,\dots,\lambda'_{p^\ell r})$. As
$$
\pi(r,s,p)=\prod_{i=1}^k\Rev(t_{i-1}+1,t_i),
$$
we have $n^{\pi(r,s,p)}=t_{i-1}+t_i+1-n$ for each $n$ such that $t_{i-1}+1\le n\le t_i$. This implies that $\lambda_n=r+s-t_{i-1}-t_i$ for each such $n$ by Definition~\ref{D:pi}, and this holds for each $i$. Then from~\cite{GPX1}*{Theorem~5}, for each $i\le k$,  we deduce that $\lambda'_j=p^\ell r+p^\ell s-p^\ell t_{i-1}-p^\ell t_i$ for $p^\ell t_{i-1}+1\le j\le p^\ell t_i$. Thus, again by Definition~\ref{D:pi}, we obtain
$$
j^{\pi(p^\ell r,p^\ell s,p)}=p^\ell r+1-j+p^\ell s-(p^\ell r+p^\ell s-p^\ell t_{i-1}-p^\ell t_i)=p^\ell t_{i-1}+p^\ell t_i+1-j
$$
where $p^\ell t_{i-1}+1\le j\le p^\ell t_i$, which means
$$
\pi(p^\ell r,p^\ell s,p)=\prod_{i=1}^k\Rev(p^\ell t_{i-1}+1,p^\ell t_i)
$$
as asserted.
\end{proof}

\begin{remark}\label{rem:compute}{\rm
Proposition~\ref{prop4} is effective in computing Jordan permutations, but it must be used with care. In particular,
if, for some $i$, we have $n_i=t$ and $n_{i-1}=t-1$, then by Theorem~\ref{Prop2}, the expression for $\pi(r,s,p)$ involves
the fixed point $\Rev(t,t)$, and as we remarked earlier, fixed points are usually omitted when writing a permutation in its disjoint 
cycle representation. This fixed point of $\pi(r,s,p)$ corresponds in $\pi(p^\ell r,p^\ell s,p)$ to the factor $\Rev(p^\ell(t-1)+1, p^\ell t)$.
}
\end{remark}

The `top group' in our wreath product theorem (Theorem~\ref{T:wreath})
is related to the fact
that $n^{\pi(r,s,p)}$ has a particularly simple form modulo $r_p$.

\begin{lemma}\label{lem8}
If $p^e\mid r$, then  for each $n\in[r]$, $n^{\pi(r,s,p)}\equiv s+1-n\pmod{p^e}$.
\end{lemma}

\begin{proof}
Let $\lambda_1\ge\cdots\ge\lambda_r>0$ be the parts of $\lambda(r,s,p)$.
From $p^e\mid r$, we have that $p^e$ divides $\lcm(r,s)$, and hence
$p^e$ divides $\gcd(\lambda_1,\dots,\lambda_r)$ by \cite{GPX2}*{Theorem~2}.
Thus by \eqref{eq4},
\[
  n^{\pi(r,s,p)}=r+s+1-n-\lambda_n\equiv s+1-n\pmod{p^e}\qquad\text{for each } n\in[r].\qedhere
\]
\end{proof}

\subsection{Explicit values of Norman permutations for small parameters}

We next describe $\pi(r,s,p)$ for certain small values of $r$ and certain  congruence classes of $s$. First we record some
results from \cite{GPX1} for $r\leq3$.

\begin{lemma}\label{pismall}
Let $r\in\{1, 2,3\}$, $s\geq r$, and let $p$ be a prime. Then $\pi(1,s,p)=()$, 
\[
  \pi(2,s,p)=\begin{cases}
              (1,2)&\textup{if $s_{\bmod p} = 0$}\\
              ()&\textup{otherwise,}
             \end{cases}
\]
and the values of $\pi(3,s,p)$ are as in Table~$\ref{tabr3}$,  where $e=1$ for odd $p$, and $e=2$ for $p=2$.
\end{lemma}

\begin{table}[!ht]
\caption{Values of $\pi(3,s,p)$ for Lemma~\ref{pismall}.}\label{tabr3}
\footnotesize
\begin{tabular}{cc} \toprule
$s_{\bmod p^e}\quad$ & $\quad\pi (3,s,p)$   \\ \midrule
$0$			& $(1,3)$  \\
$1$			& $(2,3)$  \\
$-1$		& $(1,2)$  \\
otherwise	& $()$  \\ \bottomrule
\end{tabular}
\end{table}

\begin{proof}
The values for $\pi(r,s,p)$, where $r\leq2$, and where $(r, s_{\bmod p^e}) = (3, 0), (3, 1)$, 
 all follow from~\cite{GPX1}*{Table~1}, Definition~\ref{D:pi}, and \eqref{eqelambda}.
The values for $(r, s_{\bmod p^e})$ equal to $(3, -1)$ then follow from Proposition~\ref{lem6}(c). 
For $r=3$,  and $s_{\bmod p^e}$ lying between $2$ and $p-2$, the 
corresponding Jordan partition $\lambda(3,s,p)$ is standard, by \cite{GPX1}*{Theorem 2} for odd $p$, 
and by~\cite{GPX1}*{Table~1} for $p=2$.  Hence  in these cases $\eps(3,s,p) = (2,0,-2)$ 
and $\pi(3,s,p)=()$.
\end{proof}

Our next result determines $\pi(r,s,p)$ for all $r$, for certain congruences of $s$ modulo a $p$-power $p^m\geq r$. 
It is a consequence of Propositions~3, 13 and 14
of~\cite{GPX1}, and Definition~\ref{D:pi}.

\begin{proposition}\label{prop1}
  Let $1\le r\le s$, let $p$ be a prime, and let $m$ be such that $r\leq p^m$. 
If  $s_{\bmod p^m}$ is as in the second column of Table~$\ref{tabssmall}$,
then $\pi(r,s,p)$ is given by the third column of Table~$\ref{tabssmall}$.
\end{proposition}

\begin{table}[!ht]
\caption{Norman involutions for Proposition~\ref{prop1} where $s\equiv 0, 1, 2,3\pmod{p^m}$.}\label{tabssmall}
\footnotesize
\begin{tabular}{cccc} \toprule
\textup{Case}&	$s_{\bmod p^m}$ & $\pi(r,s,p)$  & Conditions on $r,s$  \\ \midrule
$0$	& 0					& $\Rev(1,r)$	&  $1\le r\le s$  \\
$1$	& 1					& $\Rev(2,r)$	&  $2\le r\le s$  \\
$2$	& 2					& $(1,2)\circ\Rev(3,r)$		&  $3\le r\le s$ and $r_{\bmod p} = 0$  \\
       &  					& $\Rev(3,r)$	&  $3\le r\le s$ and $r_{\bmod p} \ne 0$ \\
$3$	& 3					& $\pi(3,r,p)\circ\Rev(4,r)$	& $4\le r\le s$  \\ \bottomrule
\end{tabular}
\end{table}

\begin{proof}
It follows from Lemma~\ref{lem6}~(b) that we may assume that
$m=\lceil\log_pr\rceil$.  In   Case~0, it follows from~\cite{GPX1}*{Proposition~3}
that $\eps(r,s,p)$ is the zero vector and hence that $\pi(r,s,p)=\Rev(1,r)$. 
Case~1 follows  from Proposition~\ref{T:s1}(a) as $\pi(1,r,p)\in\Sy_1$ is trivial.
  In Case~2, we have $3\le r\le s$ and $s\equiv2\pmod{p^m}$. Therefore, by 
  Proposition~~\ref{T:s1}(a),   $\pi(r,s,p)=\pi(2,r,p)\circ\Rev(3,r)$,
  and by  Lemma~\ref{pismall}, $\pi(r,s,p)$ is as in Table~\ref{tabssmall}. 
  Similarly, in Case~3, we have $4\le r\le s$ and $s\equiv3\pmod{p^m}$, and hence
  $\pi(r,s,p)=\pi(3,r,p)\circ\Rev(4,r)$ by Proposition~\ref{T:s1}(a).
\end{proof}

\begin{remark}
(a) For case $3$ of Table~\ref{tabssmall} we can determine $\pi(r,s,p)$ precisely using the values for 
$\pi(3,r,p)$ from Table~\ref{tabr3}, but we refrained from expanding this line of the Table~\ref{tabssmall} into four lines.

(b) We can also determine $\pi(r,s,p)$ when $s\equiv -1,-2,-3\pmod {p^m}$.
To do this we use Proposition~\ref{lem6}(c). For example, if   $s\equiv -2\pmod {p^m}$, then
this application yields
$\pi(r,s,p)=\Rev(1,r-2)\circ\pi(2,r,p)^{\Rev(1,r)}$.
\end{remark}

Finally in this subsection we determine in Proposition~\ref{prop5} several kinds of Norman involutions which 
we use in the final section to prove Theorem~\ref{T:wreath}. It requires the following information 
from the Green ring.

\begin{lemma}\label{green}
Let $p$ be a prime and let $b=p^e\geq 1$.
\begin{enumerate}[{\rm (a)}]
  \item For $b\geq1$, $V_b\otimes V_1 = V_b$, and for $b>1$, 
  $V_{b+1}\otimes V_b = V_{2b} \oplus (b-1) V_b$.
  \item If $r$ has $p$-part $b>1$, then 
  $V_{b+1}\otimes V_r = V_{r+b} \oplus (b-1) V_r\oplus V_{r-b}$.
\end{enumerate}
\end{lemma}

\begin{proof}
(a) Since $\lambda(1,b,p)$ has only one part, it equals $(b)$ and so  $V_1\otimes V_b = V_b=V_b\otimes V_1$
holds for each $b\geq1$.
For the second equality assume that $b>1$. Since $b<b+1 < p^{e+1}$, we may apply Renaud's algorithm  as stated
  in~\cite{GPX2}*{Proposition~6} with  $r=r_0p^e+r_1$, $s=s_0p^e+s_1$,
  where $(r_0,r_1,s_0,s_1)=(1, 0, 1, 1)$. 
Now $V_{r_1}\otimes V_{s_1} = V_0\otimes V_1 = V_0$ by the first equality just proved. If $p>2$, then in~\cite{GPX2}*{Proposition~6} we have $(c,d_1,d_2) = (0, 1, 1)$ and we obtain 
$V_{b}\otimes V_{b+1} = V_{2p^e} \oplus (p^e-1) V_{p^e} = V_{2b} \oplus (b-1) V_{b}$.
If $p=2$, then  $(c,d_1,d_2) = (1, 0, 1)$ and~\cite{GPX2}*{Proposition~6}
gives the desired decomposition.

(b) Write $r=ab$ with $b=r_p$, and $a$ coprime to $p$. The case $a=1$ follows from part (a), as $V_0=0$. Suppose now that $a>1$. Since $b>1$ we have $b+1<p^e+1$ and $b+1<r$. To prove part (b) it is sufficient to 
prove that $\lambda(b+1,r,p)$ is the $(b+1)$-tuple $(r+b, r, \dots, r, r-b)$, or equivalently, by \eqref{eqelambda},
that the deviation vector $\eps(b+1,r,p)$  is the $(b+1)$-tuple $(b, 0, \dots, 0, -b)$.
Note that  $(b, 0, \dots, 0, -b)$ is equal to its `negative reverse' (see~\cite{GPX1}*{Definition~1}).
If $r'$ is such that $b+1<r'$ and $r+r'\equiv 0\pmod{p^{e+1}}$, then by \cite{GPX1}*{Theorem 4}, 
 $\eps(b+1,r',p)$ is the negative reverse of  $\eps(b+1,r,p)$. Thus to prove part (b) it is sufficient 
 to prove that  $\eps(b+1,r',p) = (b, 0, \dots, 0, -b)$. 
 Since $r_{\bmod p^{e+1}} = (ab)_{\bmod bp} = (a_{\bmod p}) b$, we may take $r'=a'b$ 
 where $2\leq a'\leq p+1$ and $a+a'\equiv 0\pmod{p}$. As $a_{\bmod p}\ne 0$, it follows that $a'\ne p$.
 
  We claim that  
  $V_{b+1}\otimes V_{a'b} = V_{(a'+1)b}\oplus(b-1)V_{a'b}\oplus V_{(a'-1)b}$
  for $2\leq a'\le p+1$.
  We consider three cases: (i) $2\le a'\le p-2$, (ii) $a'=p-1$, 
  and (iii) $a'=p+1$.

  {\sc Case}~(i) $2\le a'\le p-2$. Here we use
  Renaud's algorithm~\cite{GPX2}*{Proposition~6} with $b+1=r_0 p^e+r_1$ and $a'b=s_0 p^e+s_1$
  where $(r_0,r_1,s_0,s_1)=(1,1,a',0)$. In this case
  $V_{r_1}\otimes V_{s_1}=V_1\otimes V_0=V_0$, so the parameter $\ell$ in
  ~\cite{GPX2}*{Proposition~6} is  $0$. Since $r_0+s_0=1+a'<p$,
  the triple $(c,d_1,d_2)$ equals $(0,r_0,r_0)=(0,1,1)$ in the notation
  of~\cite{GPX2}*{Proposition~6}, and the desired decomposition holds.

  {\sc Case}~(ii) $a'=p-1$. Here we have the same choices for $(r_0,r_1,s_0,s_1)$,
  but in this case $r_0+s_0=p$, so $(c,d_1,d_2)=(1,0,1)$ and
  $V_{b+1}\otimes V_{(p-1)b}=V_{pb}\oplus(b-1)V_{(p-1)b}\oplus V_{(p-2)b}$, as desired.

  {\sc Case}~(iii) $a'=p+1$. In this case $a'b = (p+1)p^e < p^{e+2}$ and we compute   
 $V_{b+1}\otimes V_{a'b}$ using Renaud's algorithm applied to the decompositions
 $b+1=0\cdot p^{e+1} + (b+1)$ and $a'b= 1\cdot p^{e+1} + p^e$, so the parameters 
 of~\cite{GPX2}*{Proposition~6} are   $(r_0,r_1,s_0,s_1)=(0,b+1, 1, b)$
 and $(c, d_1, d_2) = (0,0,0)$. Also  
  $V_{r_1}\otimes V_{s_1}=V_{b+1}\otimes V_{b}$ which 
  by part (a) is equal to $V_{2b}\oplus (b-1) V_b$.
Therefore by~\cite{GPX2}*{Proposition~6},
\[
V_{b+1}\otimes V_{(p+1)b}=V_{p^{e+1}}\oplus V_{p^{e+1}+2b}\oplus (b-1)V_{p^{e+1}+b}
= V_{(p+2)b}\oplus (b-1) V_{(p+1)b}\oplus V_{pb}.
\]
  This completes the proof that  $V_{b+1}\otimes V_{a'b}=V_{(a'+1)b}\oplus(b-1)V_{a'b}\oplus V_{(a'-1)b}$ for $2\le a'\le p+1$. 
  It follows from equation~\eqref{eqelambda}, that 
 $\eps(b+1,r',p) = (b, 0, \dots, 0, -b)$. However, as explained above
 $\eps(b+1,r',p) = \eps(b+1,r,p)$, and so part~(b) now follows.  
 \end{proof}

\begin{proposition}\label{prop5}
  Let $p$ be a prime and let $r$ be a positive integer with $p$-part $b$
  where $1<b<r$. Let $m$ be such that $r\leq p^m$. Then 
  \begin{align*}
  \pi(r,p^m+b,p)	&=\Rev(1,b)\,\Rev(b+1,r),	\\
  \pi(r,p^m+2b,p) 	&=\Rev(1,b)\,\Rev(b+1,2b)\,\Rev(2b+1,r)&&\textup{if $2b<r$,}	\\
  \pi(r,p^m+b+1,p)	&=\Rev(2,b)\,\Rev(b+2,r).&&	
  \end{align*}
  Moreover, in the Green ring,
   \begin{equation}\label{E:prop5}
    V_r\otimes V_{p^m+b+1}=V_{p^m+r+b}\oplus(b-1)V_{p^m+r}\oplus V_{p^m+r-b}
                          \oplus(r-b-1)V_{p^m}.
  \end{equation} 
\end{proposition}

\begin{proof}
Write $r=ab$ where $a\geq1$ is coprime to $p$ and $b=p^e$. Since $r=ab\leq p^m$, we have $m=e+\ell$ where $a\leq p^\ell$.
By Proposition~\ref{prop1}, $\pi(a,p^\ell+1,p)=\Rev(2,a)$. 
Write $\pi(a,p^\ell+1,p)$ as $\Rev(1,1)\Rev(2,a)$. Then 
Proposition~\ref{prop4} implies that 
\[
  \pi(r, p^m+b,p)=\pi(ab, p^\ell b+b,p)  =\Rev(1,b)\Rev(b+1,r),
\]
establishing the first equality. 

Next, by Proposition~\ref{prop1}, $\pi(a,p^\ell+2,p)=\Rev(3,a)$  since $a_{\bmod p}\ne 0$. 
Then by Proposition~\ref{prop4}, $\pi(r,p^m+2b,p)=\Rev(1,b)\Rev(b+1,2b)\Rev(2b+1,r)$, 
and the second equality is proved.

The third equality will follow from \eqref{E:prop5}, for equation \eqref{E:prop5} is equivalent to 
  \[
    \lambda(r,p^m+b+1,p)=
    (p^m+r+b,\overbrace{p^m+r,\dots,p^m+r}^{b-1},p^m+r-b,\overbrace{p^m,\dots,p^m}^{r-b-1}),
  \]
 and it then follows from \eqref{eq4} that the corresponding Norman permutation $\pi(r,p^m+b+1,p)$
 is equal to $\Rev(2,b)\,\Rev(b+2,r)$, fixing the points $1$ and $b+1$. Thus
 it remains  to prove \eqref{E:prop5}. We do this using Renaud's algorithm as stated
  in~\cite{GPX2}*{Proposition~6} to compute $V_r\otimes V_s$, where $s=p^m+b+1$, 
  from a smaller known decomposition of   $V_{r_1}\otimes V_{s_1}$. Here we 
 have $1\le r\le s\le p^{m+1}$ and we write $r=r_0p^m+r_1$, $s=s_0p^m+s_1$
  where $(r_0,r_1,s_0,s_1)=(0,r,1,b+1)$. We have 
   $V_{r_1}\otimes V_{s_1} = V_{b+1}\otimes V_r = V_{r+b} \oplus (b-1) V_r\oplus V_{r-b}$
by Lemma~\ref{green}(b).  Renaud's algorithm now gives the decomposition \eqref{E:prop5} for $V_r\otimes V_s$,
completing the proof.
 \end{proof}

\section{Trivial Norman permutations: proof of Theorem~\ref{newBarry-main}}
\label{s:triv}

Here we prove Theorem~\ref{newBarry-main} by establishing
the following more general technical result. 

\begin{proposition}\label{newBarry}
  If $1\le r\le s$ and $p$ is a prime, then the following are equivalent:
\begin{enumerate}[{\rm (i)}]
  \item $\lambda(r,s,p)$ is a standard partition;
  \item $\pi(r,s,p)$ is the identity permutation;
  \item $(r,s,p)$ is a standard triple as in \textup{Definition~\ref{stdtriple}};
  \item $L(1)=L(2)=\cdots=L(r)=1$;
  \item $R(1)=R(2)=\cdots=R(r)=0$;
  \item $\delta_1=\delta_2=\cdots=\delta_r=1$.
\end{enumerate}
\end{proposition}

The proof of Proposition~\ref{newBarry} needs the following two lemmas.
Recall that, for an integer $y$
and prime power $q$, $y_{\bmod q}$ denotes the integer in $[0, q-1]$ such that 
$y\equiv y_{\bmod q}\pmod{q}$.

\begin{lemma}\label{lem11}
Let $p$ be an odd prime, and let $r, s$ be integers such that
$r\ge 2$ and $s\geq2$. Let $t$ be the unique integer in the
interval $[r,p+r-1]$ such that $t_{\bmod p} = s_{\bmod p}$, and let
\[
  S=\{(k,d)\mid2\le k\le d\le p+1-k\}\cup\{(k,p+k-1)\mid2\le k\le(p+1)/2\}.
\]
Then $(r,t)\in S$ if and only if $(s-r+1)_{\bmod p} + 2r\le p+2$.
\end{lemma}

\begin{proof}
Since $t\equiv s\pmod{p}$ and $r\leq t\leq p+r-1$, it follows that
$s-r+1 \equiv t-r+1\pmod{p}$ and $1\leq t-r+1\leq p$. Thus either
\begin{enumerate}[{\rm (i)}]
 \item $s\not\equiv r-1\pmod{p}$ and $(s-r+1)_{\bmod p} = t-r+1$, or
 \item $s \equiv r-1\pmod{p}$, $t=p+r-1$, and $(s-r+1)_{\bmod p} = 0$.
\end{enumerate}

Suppose first that $(r,t)\in S$. It follows from the definition of
$S$ that $2r\le p+1$.
In case (i), we have $(s-r+1)_{\bmod p} + 2r = t+r+1$ and this is at
most $p+2$ since,  in case (i), $(r,t)\in S$ implies that $t\leq p+1-r$.
In case (ii), $(s-r+1)_{\bmod p} + 2r = 2r < p+2$.

Suppose conversely that $(s-r+1)_{\bmod p} + 2r \leq p+2$.
In case (i), the integer $t$ is equal to $(s-r+1)_{\bmod p} + r - 1$
which by assumption is at most $p+1-r$, and hence $2\le r\le t\le p+1-r$.
Thus $(r,t)\in S$.
In case (ii), $2r = (s-r+1)_{\bmod p} + 2r \leq p+2$, and as $p$ is odd,
this gives $2r\leq p+1$, so $2\leq r\leq \frac{p+1}{2}$. Since in
this case $t=p+r-1$ we have $(r,t)\in S$. Thus the lemma is proved.
\end{proof}

\begin{lemma}\label{lem12}
Let $p$ be an odd prime, and let $r, s$ be integers such that
$r > p$ and $s\geq2$. Let $m=\lceil\log_pr\rceil$ (so $m > 1$),
and let $t$ be the unique integer in the interval $[r,p^m+r-1]$
such that $t_{\bmod p^m} = s_{\bmod p^m}$.
Further, let $a = r_{\bmod p^{m-1}}$, $b = s_{\bmod p^{m-1}}$,
and $S=(T_1\setminus T_2)\cup T_3$, where
\begin{align*}
 T_1 &=\left\{\left(ip^{m-1}+\frac{p^{m-1}+\delta}{2},\ jp^{m-1}+\frac{p^{m-1}+\delta'}{2}\right)\right.\\
     &\hskip3cm \big|\ 1\le i\le\frac{p-1}{2},\quad i\le j\le p-i-1,\quad
	\delta,\delta'\in\{1,-1\}  \left.\right\},\\
T_2&=\left\{\left(ip^{m-1}+\frac{p^{m-1}+1}{2},\ ip^{m-1}+\frac{p^{m-1}-1}{2}\right)
\bigm|\ 1\le i\le\frac{p-1}{2}\right\},\\
T_3&=\left\{\left(ip^{m-1}+\frac{p^{m-1}+1}{2},\ ip^{m-1}+\frac{p^{m-1}-1}{2}+p^m\right)
\big|\ 1\le i\le\frac{p-1}{2}\right\},
\end{align*}
  Then, setting $h=(p^{m-1}-1)/2$, a pair $(r,t)$ lies in $S$ if and only if
\begin{equation}\label{eq9}
a-h,b-h\in\{0,1\}
\text{ and \ }\left\lfloor\frac{(s-r+1)_{\bmod p^m}}{p^{m-1}}\right\rfloor+2\left\lfloor\frac{r}{p^{m-1}}\right\rfloor\le p-1.
\end{equation}
\end{lemma}

\begin{proof}
From the definition of $t$ we have $r\le t$ and so $(r,t)\notin T_2$. Thus, to prove the lemma,
it suffices to prove that $(r,t)\in T_1\cup T_3$ if and only if~\eqref{eq9} holds.

Suppose first that $(r,t)\in T_1\cup T_3$. Then from the definitions of the
sets $T_1$ and $T_3$, we have $a-h,b-h\in\{0,1\}$.
Note that, since $t_{\bmod p^m}=s_{\bmod p^m}$, we also have
$t_{\bmod p^{m-1}}=s_{\bmod p^{m-1}}=b$.
If $(r,t)\in T_1$, then $i:=\lfloor r/p^{m-1}\rfloor$ and
$j:=\lfloor t/p^{m-1}\rfloor$ are such that $1\leq i\leq j\le p-i-1$, and
$t-r+1=(j-i)p^{m-1}+b-a+1$, with $0\le b-a+1\le2$.
By the definition of $t$, $1\le t-r+1\le p^m$. Moreover $j-i\le p-2i-1\le p-3$ and hence $t-r+1<p^m$. It follows that
$(s-r+1)_{\bmod p^m}=(t-r+1)_{\bmod p^m}=t-r+1$.
Thus
\[
  \left\lfloor\frac{(s-r+1)_{\bmod p^m}}{p^{m-1}}\right\rfloor+2\left\lfloor\frac{r}{p^{m-1}}\right\rfloor=
  \left\lfloor\frac{t-r+1}{p^{m-1}}\right\rfloor+2i=(j-i)+2i\le p-1.
\]
If $(r,t)\in T_3$, then $i:=\lfloor r/p^{m-1}\rfloor$ satisfies $1\leq i\le(p-1)/2$,  and we have  $t-r+1=p^m$, whence
\[
  \left\lfloor\frac{(s-r+1)_{\bmod p^m}}{p^{m-1}}\right\rfloor+2\left\lfloor\frac{r}{p^{m-1}}\right\rfloor=
  \left\lfloor\frac{(t-r+1)_{\bmod p^m}}{p^{m-1}}\right\rfloor+2i=2i\le p-1.
\]

Conversely suppose  that \eqref{eq9} holds. By the definitions of $a$ and $b$, and since $t\equiv s\pmod{p^m}$,
we have $r=ip^{m-1}+a$ and $t=jp^{m-1}+b$, for some $i, j$. Then since $p^{m-1}<r\le t\le p^m+r-1$, we have $1\le i\le j$
and $j-i\le p$. If $j-i<p$, then
\[
  \left\lfloor\frac{(s-r+1)_{\bmod p^m}}{p^{m-1}}\right\rfloor+2\left\lfloor\frac{r}{p^{m-1}}\right\rfloor=
  \left\lfloor\frac{(t-r+1)_{\bmod p^m}}{p^{m-1}}\right\rfloor+2i=(j-i)+2i=j+i,
\]
and by \eqref{eq9} the pair $(r,t)\in T_1$. On the other hand, if $j-i=p$, then the condition
$t\le p^m+r-1$ forces $a=(p^{m-1}+1)/2$ and $b=(p^{m-1}-1)/2$, and hence the equality
\[
  \left\lfloor\frac{(s-r+1)_{\bmod p^m}}{p^{m-1}}\right\rfloor+2\left\lfloor\frac{r}{p^{m-1}}\right\rfloor=
  \left\lfloor\frac{(t-r+1)_{\bmod p^m}}{p^{m-1}}\right\rfloor+2i=2i.
\]
In conjunction with \eqref{eq9} this shows that $(r,t)\in T_3$. This completes the proof.
\end{proof}

We note the following arithmetic fact
(whose proof is straightforward and is omitted).

\begin{lemma}\label{lem-cong}
 Let $r,s$ be integers and $p$ an odd prime such that $2\leq r\leq s$ and $r\leq (p+1)/2$.
 Then the following are equivalent.
 \begin{enumerate}[{\rm (a)}]
  \item $(s-r+1)_{\bmod p} \leq p+2-2r$.
  \item Either \textup{(i)} $(s-r)_{\bmod p} \leq p+1-2r$ or
               \textup{(ii)} $(s-r)_{\bmod p} = p-1$.
 \end{enumerate}
\end{lemma}

Now we prove Proposition~\ref{newBarry} (and hence Theorem~\ref{newBarry-main}),
using Barry's results and the lemmas above.

\begin{proof}[Proof of Proposition~\ref{newBarry}.]
Recall the definition of a standard partition in \eqref{eq:std}. By
Theorem~\ref{Prop2}(d) it follows that $\lambda(r,s,p)$ is a
standard partition if and only if
\begin{equation}\label{ELR1}
  L(n)=R(n)+1\qquad\textup{for $1\le n\le r$.}
\end{equation}
However,~\eqref{ELR1} is equivalent to $\pi(r,s,p)$ being the
identity permutation, by Theorem~\ref{Prop2}(c). Thus conditions (i), (ii), and \eqref{ELR1}
are all equivalent.

We next prove that~\eqref{ELR1} implies (iv)--(vi).
So assume that \eqref{ELR1} holds. It follows from the definition of
$L(n)$ that $L(1)=1$, and hence, by~\eqref{ELR1}, also
$R(1)=0$. This means that $\delta_1=1$ by the definition of $R(1)$.
Suppose $n$ is such that $1\le n<r$ and $L(n)=\delta_n=1$ and $R(n)=0$.
(We have proved this for $n=1$.) Then
$L(n+1)=1$ by the definition of $L$, and by~\eqref{ELR1} we have
$R(n+1)=0$, whence $\delta_{n+1}=1$.
Thus by induction it follows that, if \eqref{ELR1} holds, then all of (iv)--(vi)
hold.

By the definitions of $L$ and $R$, condition (vi) implies both (iv) and (v).
Also, since $\delta_0=\delta_r=1$, condition (iv) implies (vi), and condition (v) implies (vi).
Thus conditions (iv)--(vi) are pairwise equivalent.  Moreover, the equivalent conditions
(iv) and (v), together with Theorem~\ref{Prop2}(c) imply that $\pi(r,s,p)=1$.
We have now proved the equivalence of conditions (i), (ii), (iv), (v) and (vi).

Finally we apply Barry's results and Lemma~\ref{lem12} to show that conditions (i) and (iii) are equivalent.
If $r=1$ then the only possibility for the partition is $\lambda(1,s,p)= (s)$, which is standard, so
both (i) and (iii) hold. Assume now that $2\leq r\leq s$. If $p=2$ then, by Barry's result \cite{B}*{Theorem 1},
$\lambda(r,s,2)$ is standard if and only if Definition~\ref{stdtriple} (line 2 or 3 of Table~\ref{tab3}) holds.
Assume now that $p$ is odd, and let $m=\lceil\log_p(r)\rceil$.

Suppose first that $m=1$, that is,  $r\leq p$.
Then by \cite{B}*{Theorem 2}, $\lambda(r,s,p)$ is standard if and only if $r\leq\frac{p+1}{2}$,
and either $(s-r)_{\bmod p} \leq p+1-2r$ or $(s-r)_{\bmod p}=p-1$. By Lemma~\ref{lem-cong}, these conditions
hold if and only if $(s-r+1)_{\bmod p} \leq p+2-2r$ (the condition in line 4 of Table~\ref{tab3}).  Thus
$\lambda(r,s,p)$ is standard if and only if $(r,s,p)$ is standard.

Finally suppose that $m>1$, so $p < r\leq s$, and let $t$
be the unique integer in the interval $[r,p^m+r-1]$
such that $t_{\bmod p^m} = s_{\bmod p^m}$.  Then, by \cite{B}*{Theorem 3},
$\lambda(r,s,p)$ is standard if and only if $(r,t)$ lies in the set $S$ defined in Lemma~\ref{lem12}.
This in turn holds (applying Lemma~\ref{lem12}) if and only if line 5 of Table~\ref{tab3} holds.
Thus in this case also $\lambda(r,s,p)$ is standard if and only if $(r,s,p)$ is standard.
Therefore in all cases we have proved that conditions (i) and (iii) are equivalent.
\end{proof}

\section{The groups \texorpdfstring{$G(r,p)$}{}}\label{sec2}

In this section we prove Theorem~\ref{T:wreath} which determines the structure of the subgroups
\[
  G(r,p)=\langle\pi(r,s,p)\mid r\le s\rangle
\]
of $\Sy_r$.
Computation played a key role in establishing this theorem.
The \Magma\ code available at~\cite{Gl} was useful for both conjecturing the
truth of Theorem~\ref{T:wreath}, and for discovering the key identity
of Lemma~\ref{L9} from which we constructed our proof.
First we show that
$G(r,p)$ has a generating set of size less than $rp$.

\begin{lemma}\label{cor4}
  If $p$ is a prime and $2\leq r\le p^m$, then
  \[G(r,p)=\langle\pi(r,s,p)\mid r\le s\le r+p^m-1\rangle.\]
\end{lemma}

\begin{proof}
If $s' > r+p^m-1$ then there exists a positive integer $k$ such that 
$s:=s'-kp^m$ satisfies $r\leq s\leq r+p^m-1$. By Proposition~\ref{lem6}(b), 
$\pi(r,s,p) =\pi(r,s',p)$.
\end{proof}

\begin{conjecture}
  The generating set $\Sigma:=\{\pi(r,s,p): r\le s\le r+p^m-1\}$ for $G(r,p)$
  has at most $2r$ distinct elements for each prime $p$, and
  $\lim_{r\to\infty}|\Sigma|/(2r)=1$.
\end{conjecture}

By convention $G(1,p)=\D_1=\Sy_1$, so we assume from now on that $r\geq2$.  
We show that $G(r,p)$ preserves a
system of imprimitivity on $\Omega := [r]$, namely
\[
  {\mathcal B}=\{\Omega_1,\dots,\Omega_b\}
  \quad\textup{where}\quad
  \Omega_j=\{n\in[r]\mid n\equiv j\pmod{b}\}\quad\textup{and}\quad |\Omega_j|=a,
\]
where $a$ is the $p'$-part and $b$ is the $p$-part of $r$.
It follows from Lemma~\ref{lem8}, applied with $p^e=b$, 
that ${\mathcal B}$ is a $G(r,p)$-invariant partition of $[r]$.
Moreover, the action of $G(r,p)$ on ${\mathcal B}$ induces an action
$\phi$ of $G(r,p)$ on $[b]$ given by
\begin{equation}\label{E:phi}
  \phi\colon G(r,p)\to\Sy_b\quad\textup{where for $n\in [b]$, }\quad
   n^{\phi(\pi(r,s,p))}=(s-n)_{\bmod b}+1.
\end{equation}

The \emph{dihedral group} $\D_b$ of \emph{degree~$b$} is a subgroup of
the symmetric group $\Sy_b$ and $|\D_b|=2b$ if $b\ge3$. If $b=1,2$, then
$\D_b=\Sy_b$, so $|\D_b|=b$ in these cases.
Warning: some authors write $\D_{2b}$ for the dihedral group of
\emph{order~$2b$} for $b\ge2$.

  Consider $\Sy_a\wr\D_b$ acting
in product action on the grid~\eqref{blocks}. For each $j$, let $\Sym(\Omega_j)\cong \Sy_a$
act on the $j$th column, and let the base group
\[
  (\Sy_a)^b\cong\Sym(\Omega_1)\times\cdots\times\Sym(\Omega_b)
\]
act on the $b$
columns (or blocks). The top group $\D_b$ permutes the blocks $\Omega_1,\dots,\Omega_b$ transitively.
\begin{equation}\label{blocks}
\begin{tabular}{|c|c|c|c|} \hline
  $1$&$2$&$\cdots$&$b$\\ 
  $b+1$&$b+2$&$\cdots$&$2b$\\ 
  $\vdots$&$\vdots$&$\cdots$&$\vdots$\\ 
  $(a-1)b+1$&$(a-1)b+2$&$\cdots$&$ab$\\ \hline
\end{tabular}
\end{equation}

\begin{lemma}\label{thm2}
Let $r$ be a positive integer with $p$-part $b$ and $p'$-part $a$,
and suppose that $\phi$ is defined as in~\eqref{E:phi}.
Then 
\[
  \phi(G(r,p))=G(b,p)=\D_{b},\quad\textup{and\quad $G(r,p)\le\Sy_a\wr\D_b\le\Sy_{ab}$}.
\]
In particular, Theorem~$\ref{T:wreath}$ is proved in the case $a=1$.
\end{lemma}

\begin{proof}
If $b=1$ then  $\phi(G(r,p)=\Sy_1=1$, and  by convention $\D_1=1$, and there is nothing to prove.
So suppose that $b>1$, so that $b\geq p$.
  Consider a regular $b$-gon with vertices numbered $1,2,\dots,b$
  consecutively. Reading modulo $b$, the map
$n\mapsto -n$ is a reflection, and $n\mapsto n+s+1$ is a
rotation of $2\pi (s+1)/n$ about the center of the $b$-gon.
Hence the composite $n\mapsto s-n+1$ is a reflection fixing a vertex. Thus as $s$ varies, the maps
$n\mapsto s-n+1$ generate the dihedral group $\D_b$ of degree~$b$. Therefore
$G(b,p)=\D_{b}$. Hence, as  $s$ varies, the permutations $\phi(\pi(r,s,p))$ defined by~\eqref{E:phi}
generate $\D_{b}$. This implies, in particular, that Theorem~$\ref{T:wreath}$ is proved in the case $a=1$.
\end{proof}


From now on we assume that $a>1$, that is, that $r$ is not a power of $p$. 
The kernel of $\phi$ is a subgroup of
$\Sym(\Omega_1)\times\dots\times\Sym(\Omega_{b})\cong (\Sy_{a})^{b}$.
The \emph{diagonal subgroup} of $(\Sy_a)^b$ is the image of the monomorphism
\begin{equation}\label{E:d}
  d\colon \Sy_a\to(\Sy_a)^{b}\quad\textup{defined by}\quad
  ((i-1)b+j)^{d(\sigma)}=(i^\sigma-1)b+j
\end{equation}
where $\sigma\in\Sy_a$, $i\in[a]$ and $j\in[b]$.

\begin{lemma}\label{thm4}
Let $r$ be a positive integer with $p$-part $b$ and $p'$-part $a>1$,
and suppose that $d$ is the diagonal embedding as in~\eqref{E:d}.
Then $G(r,p)\supseteq d(\Sy_a)$. In particular 
Theorem~$\ref{T:wreath}$ is proved in the case $b=1$.
\end{lemma}

\begin{proof}
 Let $m$ be the least integer such that $r\le p^m$ and $\ell$ be the least integer such that $a\le p^\ell$. Then $m=e+\ell$ where $b=p^e$. Consider the permutations $\pi_0:=\pi(r,p^m,p)$, $\pi_b:=\pi(r,p^m+b,p)$ and $\pi_{2b}:=\pi(r,p^m+2b,p)$ in $G(r,p)$. We have $\pi_0=\Rev(1,r)$  by Proposition~\ref{prop1}, and $\pi_b=\Rev(1,b)\Rev(b+1,r)$ by Proposition~\ref{prop5}. Hence, for each $i\in[a]$ and $j\in[b]$ we have
\begin{align*}
((i-1)b+j)^{\pi_0\pi_b}&=(r+1-((i-1)b+j))^{\pi_b}\\
&=((a-i)b+b+1-j)^{\pi_b}\\
&=\begin{cases}
  b+1+r-((a-i)b+b+1-j)\quad&\textup{if $1\le i\le a-1$,}\\
  1+b-((a-i)b+b+1-j)\quad&\textup{if $i=a$,}\\
  \end{cases}\\
&=\begin{cases}
  ib+j\hskip53mm&\textup{if $1\le i\le a-1$,}\\
  j\quad&\textup{if $i=a$}.
  \end{cases}
\end{align*}
This means that $\pi_0\pi_b=d(\sigma_1)$, where $\sigma_1=(1,2,\dots,a-1,a)\in\Sy_a$. 
If $a=2$, then the result is proved since $\langle\sigma_1\rangle=\Sy_a$ and hence 
$\langle d(\sigma_1)\rangle =d(\Sy_a)$. Assume now that $a\ge3$. 
By Proposition~\ref{prop1}, $\pi(a,p^\ell+2,p)=\Rev(3,a)$  since $a_{\bmod p}\ne 0$. 
Then by Proposition~\ref{prop4}, 
\[
  \pi_{2b}=\Rev(1,b)\Rev(b+1,2b)\Rev(2b+1,r).
\]
Hence for each $i\in[a]$ and $j\in[b]$, we have
\begin{align*}
((i-1)b+j)^{\pi_b\pi_{2b}}
  &=\begin{cases}
  (1+b-((i-1)b+j))^{\pi_{2b}}\quad&\textup{if $i=1$,}\\
  (b+1+r-((i-1)b+j))^{\pi_{2b}}\hskip21mm&\textup{if $2\le i\le a$,}\\
  \end{cases}\\
  &=\begin{cases}
  (b+1-j)^{\pi_{2b}}\quad&\textup{if $i=1$,}\\
  ((a+1-i)b+b+1-j)^{\pi_{2b}}\hskip24.2mm&\textup{if $2\le i\le a$,}\\
  \end{cases}\\
  &=\begin{cases}
  1+b-(b+1-j)\quad&\textup{if $i=1$,}\\
  2b+1+r-((a+1-i)b+b+1-j)\quad&\textup{if $2\le i\le a-1$,}\\
  b+1+2b-((a+1-i)b+b+1-j)\quad&\textup{if $i=a$,}\\
  \end{cases}\\
  &=\begin{cases}
  j\quad&\textup{if $i=1$,}\\
  ib+j\hskip62.5mm&\textup{if $2\le i\le a-1$,}\\
  b+j\quad&\textup{if $i=a$},
  \end{cases}
\end{align*}
so $\pi_b\pi_{2b}=d(\sigma_2)$, where $\sigma_2=(2,3,\dots,a-1,a)\in\Sy_a$. Now since
\[
  \langle\sigma_1,\sigma_2\rangle=\langle\,\sigma_1,\sigma_1^{-1}\sigma_2\,\rangle
  =\langle\,(1,2,\dots,a-1,a),(1,2)\,\rangle=\Sy_a,
\]
we conclude that
\[
  G(r,p)\ge\langle\pi_0\pi_b,\pi_b\pi_{2b}\rangle=\langle d(\sigma_1),d(\sigma_2)\rangle=d(\langle\sigma_1,\sigma_2\rangle)=d(\Sy_a).
\]
If $b=1$, then the homomorphism
$\phi\colon G(r,p)\to\Sy_{b}$ given by~\eqref{E:phi} is trivial and so 
$G(r,p)=\ker(\phi)=(\Sy_a)^{b}=\Sy_a=\Sy_r$ as required in Theorem~\ref{T:wreath}.
\end{proof}

%
%
%
%


From now on we assume that both $a>1$ and $b>1$. We show that $G(r,p)$
contains a transposition in the base group of $\Sy_a\wr\D_b$.

\begin{lemma}\label{L9}
Let $p$ be a prime,  let $r$ be a positive integer with $p$-part $b>1$ 
and $p'$-part $a>1$, and suppose that $r\leq p^m$. 
Write $\pi_k=\pi(r,p^m+k,p)$, for $0\leq k < r$.
Then the composition $\pi_1\pi_0\pi_b\pi_{b+1}$ is equal to 
the transposition $(1,b+1) \in   \Sym(\Omega_1)$.
\end{lemma}

\begin{proof}
By Proposition~\ref{prop1}, $\pi_0=\Rev(1,r)$ and $\pi_1=\Rev(2,r)$. Also
by Proposition~\ref{prop5}, $\pi_b=\Rev(1,b)\Rev(b+1,r)$, and 
  $\pi_{b+1}=\Rev(2,b)\Rev(b+2,r)$.  
  Direct computation shows that $\pi_1\pi_0\pi_b\pi_{b+1}$ fixes $n$ if $2\le n\le b$
  or $b+2\le n\le r$, and  that $\pi_1\pi_0\pi_b\pi_{b+1}$ interchanges $1$ and $b+1$.
  Thus  $\pi_1\pi_0\pi_b\pi_{b+1}= (1,b+1)$, fixing each block $\Omega_j$ setwise,
  and interchanging the two points $1, b+1$ of $\Omega_1$.
\end{proof}

\begin{proof}[Proof of Theorem~\ref{T:wreath}]
Let $p, r, a, b$ be as in the statement.
Then Theorem~\ref{T:wreath} holds if either $a=1$ or $b=1$ by Lemma~\ref{thm2}
and~\ref{thm4}. Assume now that $a>1$ and $b>1$. Then 
 by Lemma~\ref{L9}, the transposition $(1,b+1)$ lies in $G(r,p)$
 and belongs to   $\Sym(\Omega_1)\le\ker\phi $. By Lemma~\ref{thm4}, the diagonal
  subgroup $D=\textup{diag}((\Sy_a)^b)$ is contained in $G(r,p)$, and hence 
$G(r,p)$ contains $(1,b+1)^D=\{(1,b+1)^g | g\in D\}$. Since $D$ acts on $\Omega_1$ 
as $\Sy_a$, $(1,b+1)^D$ is the set of all transpositions in $\Sym(\Omega_1)$, and in particular,
  $\langle(1,b+1)^D\rangle=\Sym(\Omega_1)$, and is contained in $G(r,p)$.
  It follows from Lemma~\ref{thm2} that each of $\Sym(\Omega_1), \dots, \Sym(\Omega_b)$ 
  lies in $G(r,p)$. Thus $(\Sy_a)^b\le G(r,p)$. Finally, by Lemma~\ref{thm2}, we conclude that 
  $G(r,p)=\Sy_a\wr\D_b$.  
\end{proof}

\noindent\textsc{Acknowledgements.}
We thank Michael Barry for drawing our attention to~\cite{B}.
The first and second authors acknowledge the support of the Australian Research Council Discovery Grant DP160102323. 
The third author's work on this paper was done when he was a research associate at the University of Western Australia supported by the Australian Research Council Discovery Project DP150101066.

\end{document}